\documentclass[12pt,reqno]{amsart} 

\usepackage[utf8]{inputenc}         
\usepackage[T1]{fontenc}            
\usepackage[english]{babel}         
\DeclareUnicodeCharacter{2460}{\textcircled{\scriptsize 1}}
\DeclareUnicodeCharacter{2461}{\textcircled{\scriptsize 2}}
\DeclareUnicodeCharacter{2462}{\textcircled{\scriptsize 3}}
\DeclareUnicodeCharacter{2463}{\textcircled{\scriptsize 4}}
\DeclareUnicodeCharacter{2464}{\textcircled{\scriptsize 5}}
\DeclareUnicodeCharacter{2465}{\textcircled{\scriptsize 6}}
\DeclareUnicodeCharacter{2466}{\textcircled{\scriptsize 7}}
\DeclareUnicodeCharacter{2467}{\textcircled{\scriptsize 8}}
\DeclareUnicodeCharacter{2468}{\textcircled{\scriptsize 9}}
\DeclareUnicodeCharacter{25A1}{\ensuremath{\square}}
\usepackage[margin=1in]{geometry}   

\usepackage{enumitem}   
\usepackage{parskip}    
\usepackage{mdframed}   
\usepackage{microtype}  

\usepackage{amsmath}     
\usepackage{amsthm}      
\usepackage{amssymb}     
\usepackage{amsfonts}    
\usepackage{thmtools}    

\usepackage{bm}
\usepackage{systeme}     
\usepackage{mathtools}   
\usepackage{mathdots}

\usepackage{graphicx}   
\usepackage{wrapfig}    
\usepackage{tikz}       
\usetikzlibrary{positioning}
\usepackage{tikz-cd}    
\usepackage{float}
\usepackage{booktabs}                   
\usepackage{array}                      

\usepackage[varqu,varl]{inconsolata} 
\usepackage{xcolor}                  
\usepackage{hyperref} 
\hypersetup{colorlinks=true, allcolors=blue} 
\usepackage{skull}
\usepackage{color}
\usepackage{soul}
\usepackage[normalem]{ulem}
\usepackage{cancel}
\usepackage{mathrsfs}

\usepackage{listings}                           

\usepackage{silence}  

\mdfsetup{skipabove=1em,skipbelow=0em} 

\mdfsetup{
	linewidth = 0.3mm,
	innertopmargin = 2mm,
	innerbottommargin = 3.5mm,
	innerleftmargin = 3mm,
	innerrightmargin = 3mm
} 

    {\begin{proof}[Solution%
        \if\relax\detokenize{#1}\relax\else\ (#1)\fi]}%
    {\end{proof}}

    {\end{proof}}

\setlistdepth{10}  

\setlist[itemize,1]{label=•}
\setlist[itemize,2]{label=-}
\setlist[itemize,3]{label=\raisebox{-.5ex}{*}}
\setlist[itemize,4]{label=•}
\setlist[itemize,5]{label=-}

\newtheorem{theorem}{Theorem}[section]

\newtheorem{corollary}[theorem]{Corollary}
\newtheorem{lemma}[theorem]{Lemma}
\newtheorem{proposition}[theorem]{Proposition}

\theoremstyle{definition}

\newtheorem*{definition*}{Definition}

\newtheoremstyle{example}{10pt}{10pt}{}{}{\bfseries}{.}{.5em}{}
\theoremstyle{example}

\theoremstyle{remark}

\def\todo#1{\relax}

\definecolor{commentgreen}{HTML}{119F57}

\newlist{steps}{enumerate}{9}
\setlist[steps]{before=\upshape\normalfont,font=\ttfamily} %
\setlist[steps,1]{leftmargin=*,labelsep=1ex,label={\arabic*.},ref={\arabic*}} %
\setlist[steps,2]{leftmargin=*,labelsep=1ex,label={\Alph*.},ref={\thestepsi.\Alph*}} %
\setlist[steps,3]{leftmargin=*,labelsep=1ex,label={\arabic*.},
  ref={\thestepsii.\arabic*}} %
\setlist[steps,4]{leftmargin=*,labelsep=1ex,label={\alph*.},
  ref={\thestepsiii.\alph*}} %
\setlist[steps,5]{leftmargin=*,labelsep=1ex,label={\arabic*.},
  ref={\thestepsiv.\arabic*}} %
\setlist[steps,6]{leftmargin=*,labelsep=1ex,label={\Alph*.},
  ref={\thestepsv.\Alph*}} %
\setlist[steps,7]{leftmargin=*,labelsep=1ex,label={\arabic*.},
  ref={\thestepsvi.\roman*}} %
\setlist[steps,8]{leftmargin=*,labelsep=1ex,label={\alph*.},
  ref={\thestepsvii.\alph*}} %
\setlist[steps,9]{leftmargin=*,labelsep=1ex,label={\arabic*.},
  ref={\thestepsviii.\arabic*}} %
\newlist{stepitems}{itemize}{1} 
\setlist[stepitems]{before=\upshape\ttfamily,label=*,leftmargin=*,labelsep=1ex}%

\usepackage[normalem]{ulem}
  \let\oldin\in
  \let\oldcdot\cdot
  \renewcommand{\in}{\ensuremath{\oldin}}
  \renewcommand{\cdot}{\ensuremath{\oldcdot}}

  \newcommand{\tand}{\text{ and }}

  \newcommand{\Z}{\ensuremath{\bbZ}} 
  \newcommand{\Q}{\ensuremath{\bbQ}} 

  \DeclareMathOperator{\Gal}{Gal}     

  \renewcommand{\it}[1]{\emph{#1}}      

  \newcommand{\order}[1]{\ensuremath{\left\vert #1 \right\vert}} 
  \renewcommand{\index}[2]{\ensuremath{[ #1 : #2 ]}}        

  \newcommand{\divs}{\ensuremath{\mathrel{\vert}}}    

  \newcommand{\sgp}{\ensuremath{\mathrel{\leqslant}}} 
  \newcommand{\ogp}{\ensuremath{\mathrel{\geqslant}}} 

  \newcommand{\qg}[2]{\ensuremath{#1/#2}} 

\newcommand{\bbQ}{\mathbb{Q}}

  \newcommand{\bbZ}{\mathbb{Z}}

  \newcommand{\mcR}{\mathcal{R}}

\DeclareMathOperator{\rank}{rank}
\newcommand\blfootnote[1]{%
  \begingroup
  \renewcommand\thefootnote{}\footnote{#1}%
  \addtocounter{footnote}{-1}%
  \endgroup
}

\title{On a Combinatorial Characterization of Abelian Group Actions}

\author{Carson Gull}
\address[Carson Gull]{Department of Mathematics, Purdue University, 150 North University Street, West Lafayette, Indiana 47907-2067}
\email{cgull4@wgu.edu}

\author{Arthur Illidge}
\address[Arthur Illidge]{Department of Mathematics, The University of Chicago, 5734 S. University Avenue, Chicago, IL 60637}
\email{aillidge@uchicago.edu}

\author{Benjamin Mudrak}
\address[Benjamin Mudrak]{Department of Mathematics, Purdue University, 150 North University Street, West Lafayette, Indiana 47907-2067}
\email{bmudrak@purdue.edu}

\author{Joshua Park}
\address[Joshua Park]{Department of Mathematics, Purdue University, 150 North University Street, West Lafayette, Indiana 47907-2067}
\email{wljoshpark@gmail.com}

\author{Shaver Phagan}
\address[Shaver Phagan]{Bolin Science Hall, Midwestern State University, Wichita Falls, TX 76308}
\email{shaver.phagan@msutexas.edu}

\begin{document}
\maketitle


\section*{Abstract}
Let $G$ be a finite abelian group of order $n$ and $X$ a finite $G$-set. We give a necessary and sufficient condition on \(G\) for any \(X\) to be determined up to isomorphism by certain combinatorial data. Specifically, if \(M_H^X\) is the number of \(H\)-orbits in \(X\), and \(a_d^X\) is the number of \(G\)-orbits in \(X\) of cardinality \(d\), we characterize when equalities $M_H^{X_1} = M_H^{X_2}$ and $a_d^{X_1}=a_d^{X_2}$ for all $H \leq G$ and $d|n$, respectively, force an isomorphism $X_1 \simeq X_2$. Our main result is that this occurs if and only if $G$ is cyclic or $G$ is the direct sum of a finite cyclic group and a cyclic group of prime order. We conclude with an arithmetic application to $L$-function factorization.

\section{Introduction}
The study of group actions on finite sets through the lens of combinatorics and representation theory has a rich history \cite{plesken_counting_1982, kerber_applied_1999, bartel_brauer_2015} and has found powerful applications in a variety of fields, especially number theory and geometry \cite{perlis_equation_1977, sunada_riemannian_1985, leininger_length_2007, prasad_refined_2017}. The map from the Burnside ring of \(G\) to the rational representation ring of \(G\)
\begin{equation}\label{eqn:bside}
f:\Omega G\rightarrow\mathcal{R}_\Q G:X\mapsto\Q[X], 
\end{equation}
which takes a finite \(G\)-set (up to isomorphism) to its associated rational permutation representation (up to isomorphism), is of special interest to this end. Indeed, the Gassmann triples used in \cite{perlis_equation_1977} and \cite{sunada_riemannian_1985} are Brauer relations in \(\ker(f)\) given by a difference of transitive \(G\)-sets.  On the other hand, \(f\) is injective when \(G\) is cyclic, and a combinatorial explanation of this injectivity has proven useful in making explicit certain previously non-constructive results in algebraic number theory.  To articulate this, we use quantities naturally associated to a finite group action.
\begin{definition*}
    Let \(G\) be a finite group of order \(n\), acting on a finite set \(X\).  Given a divisor \(d\) of \(n\), let \(a_d^X\) denote the number of orbits of cardinality \(d\). Given a subgroup \(H\) of \(G\), set \(M_H^X=|H\backslash X|\). 
\end{definition*}

We will often omit the superscript when there is only one \(G\)-set under consideration.

\begin{theorem}[\cite{phagan_constructive_2026}\label{thm:phagan} Theorem 2.3]
    If \(G\) is a finite cyclic group of order \(n\) and \(X\) a finite \(G\)-set, then
    \begin{equation}\label{eqn:phagan}
     a_m=\sum\limits_{d|\frac{n}{m}}\frac{\mu(d)}{\varphi(md)}\sum\limits_{\delta|md}\mu(\delta)M_{G^{\frac{md}{\delta}}}, 
    \end{equation}
    where \(\varphi\) is Euler's totient, \(\mu\) is the M\"obius function, and \(G^d\) is the index \(d\) subgroup of \(G\).
\end{theorem}  

 Theorem \ref{thm:phagan} implies \(f\) is injective when \(G\) is cyclic.  
 Indeed, \(f(X)\) determines the right-hand side of Equation \eqref{eqn:phagan}, and the left-hand side of Equation \eqref{eqn:phagan} determines \(X\).  
 Furthermore, by using Theorem \ref{thm:phagan} and some Galois theory, one obtains a completely algebraic explanation of the strong multiplicity one property of the Dedekind zeta function.  
 This had previously been understood analytically through the functional equation (cf. \cite{perlis_equation_1977} Theorem 1, \cite{phagan_constructive_2026} Theorem 3.2).

It is natural to wonder how Theorem \ref{thm:phagan} might generalize to arbitrary finite abelian groups, and what applications such a generalization might have.  Specifically, for which finite abelian groups \(G\) do the quantities \(M_H\) and \(a_d\) suffice as complete invariants for \(G\)-sets?  We give a complete answer to this question with the main theorem of this note, and we give an arithmetic application in Section \ref{sec:conclusions}.
\begin{theorem}\label{thm:main}
    Knowledge of \(\left(M_H^X\right)\) and \(\left(a_d^X\right)\) is sufficient to determine any \(G\)-set $X$ up to isomorphism if and only if 
    \begin{enumerate}
        \item $G$ is cyclic; or
        \item there is a positive integer \(n\) and a prime \(p\) such that $G\simeq C_{n}\oplus C_{p}$,
    \end{enumerate}
    where \(C_n\) is the cyclic group of order \(n\).
\end{theorem}

The proof proceeds in three steps.  First, in Section \ref{sec:structure}, we identify a class of finite abelian groups \(G\) whose \(G\)-sets admit the desired combinatorial characterization. Then, in Section \ref{sec:obstruction}, we equate the uniformity of our characterization to the triviality of a certain kernel.  The rank of this kernel is computed in Theorem \ref{thm:kerrank}, and as a corollary we find that all finite abelian groups \(G\), besides those identified in Section \ref{sec:structure}, possess \(G\)-sets, which do not admit our desired combinatorial characterization.  We suspect Theorem \ref{thm:kerrank} may be of independent algebro-combinatorial interest.
We conclude in Section \ref{sec:conclusions} by applying Theorem \ref{thm:main} to study a rigidity property for certain factorizations of Artin \(L\)-functions.

\section*{Notation and Conventions}

The Burnside ring $\Omega G$ of a finite group $G$ is the free abelian group generated by isomorphism classes of finite $G$-sets, modulo the relations $X_1+X_2 = X_1\sqcup X_2$, with multiplication given by the Cartesian product $X_1\cdot X_2 = X_1\times X_2$.
It is well known that the stabilizers of any two elements in the same transitive $G$-set are conjugate subgroups. 
All groups considered in this note are abelian, so transitive $G$-sets $\qg{G}{H}$ and $\qg{G}{K}$ are isomorphic if and only if $H=K$.
In particular, the map $\qg{G}{H} \mapsto H$, from isomorphism classes of transitive $G$-sets to subgroups of $G$, is a bijection.  We follow the convention in \cite{bartel_brauer_2015} and use this bijection to write an element $X\in \Omega G$ as a formal linear combination 
$$ X=\sum_{H\sgp G}\alpha_H H, $$
where $\alpha_H$ counts the number of $G$-orbits in $X$ isomorphic to $G/H$, when \(X\) is a \(G\)-set.
Characterizing a \(G\)-set $X$ therefore reduces to finding $\alpha_H$ for each $H\sgp G$.

We make use of the number-theoretic function \(\Omega(n)\), which counts prime divisors of $n$ with multiplicity, and the $p$-adic valuation function $\nu_p(n)$, both over the domain \(\mathbb{N}_+\). 
Whether \(\Omega\) denotes the Burnside ring $\Omega G$ or the function $\Omega(n)$ should be clear from context. 
Finally, for a given $G$-set $X$, we denote the family of all $M_H^X$ for subgroups $H$ of $G$ as $\left(M_H^X\right)$ and the family of all $a_d^X$ for $d$ dividing the order of $G$ as $\left(a_d^X\right)$.

\section*{Acknowledgments}  
This work began as a PXML project at Purdue University.  
The authors would like to thank Thomas Sinclair for his support and for organizing the PXML program.  The authors would also like to thank Ilya Shkredov for his thoughtful feedback on this work.

\section{A Structure Admitting the Combinatorial Characterization}\label{sec:structure}

In Theorem \ref{thm:final-condition}, we find a condition on the subgroups of $G$, which guarantees all \(G\)-sets admit our desired combinatorial characterization.  
This condition greatly restricts the structure of \(G\) (see Proposition \ref{prop:noncyclic}).

\begin{theorem} \label{thm:final-condition}
    Let $G$ be a finite abelian group such that, for each $d$ dividing $\order{G}$, $G$ contains at most one non-cyclic subgroup of order $d$. Then finite $G$-sets are uniquely determined up to isomorphism by $\left(M_H\right)$ and $\left(a_d\right)$.
\end{theorem}

\begin{proof}
Let $X$ be a finite $G$-set.
We proceed by reverse induction on \(m\), where \(\alpha_K\) is known if \(\Omega(\order{K})\geq m\).

If $\Omega(\order{K}) = \Omega(\order{G})$, then $K = G$, so that 
$$ \alpha_K=\alpha_G = a_1. $$  
Let $K$ be a subgroup with $\Omega(\order{K})=m-1$, and assume by induction that $\alpha_{K'}$ is known for any subgroup \(K'\) with \(\Omega(|K'|)\geq m\). 
We want to find $\alpha_K$.

The set \(X^K\) of fixed points in \(X\) under the induced \(K\)-action is given by 
$$ X^K = \sum_{K' \ogp K}\alpha_{K'}K', $$ 
so 
$$ \order{X^K}=\sum_{K' \ogp K}\alpha_{K'}[G:K']=\alpha_K+\sum_{K' > K}\alpha_{K'}[G:K']. $$
By hypothesis, we know \(\alpha_{K'}\) for each subgroup \(K'\) of \(G\), such that \(K< K'\).  
Furthermore, we know \(|X^K|\) for cyclic $K$, since knowledge of \(M_H\) for each subgroup \(H\) of \(K\) allows us to completely recover the \(K\)-set structure on \(X\), by Theorem \ref{thm:phagan}. 
Hence for \(K\) a cyclic group with \(\Omega(\order{K})=m-1\), the induction hypothesis allows us to compute \(\alpha_K\).  
Now, if \(K\) is non-cyclic, then it is the unique non-cyclic \(G\)-subgroup of order \(\order{K}\), so that 
$$ a_{\order{G}/\order{K}} = \sum_{\substack{K' \sgp G\\ \order{K'} = \order{K}}}\alpha_{K'}=\alpha_K + \sum_{\substack{C \sgp G \text{ cyclic}\\ \order{C} = \order{K}}}\alpha_C, $$
and we can solve for $\alpha_K,$ since we know \(a_{\order{G}/\order{K}}\) and the \(\alpha_C\).
Hence $X$ is uniquely identifiable up to isomorphism.
\end{proof}

\begin{proposition} \label{prop:noncyclic}
A non-cyclic finite abelian group \(G\) is isomorphic to \(C_{n}\oplus C_{p}\) if and only if $G$ has at most one non-cyclic subgroup of order $d$ for each $d$ dividing the order of $G$.
\end{proposition}
\begin{proof}
Suppose $G$ has at most one non-cyclic subgroup of any given order.
Given (not necessarily distinct) primes \(q,q'\), observe that the groups \(C_{q}\oplus C_{q}\oplus C_{q}\) and \(C_{qq'}\oplus C_{qq'}\) both have at least two distinct non-cyclic subgroups sharing an order. In particular, \(G\) must have \(2\) invariant factors, and one of those invariant factors must be of prime order.  In other words, $G\simeq C_n\oplus C_p$ for some $n$ and $p$. 

For the converse, suppose \(G\) admits a decomposition \(G=A\oplus B\), with \(A\simeq C_{n}\) and \(B\simeq C_{p}\), and let $H$ be a non-cyclic subgroup of $G$ with order $d$. 
The intersection $H\cap B$ is non-trivial, otherwise $H$ would be cyclic, since \(H/(H\cap B)\) is isomorphic to a subgroup of \(A\).
It follows that $H$ contains $B$, and thus \(H = (H\cap A)\oplus B\).
Comparing orders yields $H\cap A = A_{d/p}$, the unique subgroup of \(A\) of order \(d/p\).
Notice $B$ and $A_{d/p}$ depend only on $d$, whence $H$ is the unique non-cyclic subgroup of order $d$.
\end{proof}

\begin{corollary}
    If \(G\simeq C_n\oplus C_p\), then every \(G\)-set \(X\) is uniquely determined up to isomorphism by $\left(M_H^X\right)$ and $\left(a_d^X\right)$.
\end{corollary}
\begin{proof}
    This follows immediately from Theorem \ref{thm:final-condition} and Proposition \ref{prop:noncyclic}.
\end{proof}

\section{An Obstruction to the Combinatorial Characterization}\label{sec:obstruction}

In Lemma \ref{lem:kertrivial}, we exhibit Brauer relations obstructing our combinatorial characterization of \(G\)-sets.  
This obstruction is naturally described by the kernels of the homomorphisms
$$ f:\Omega G \rightarrow \mcR_{\Q}(G): X\mapsto \Q[X]$$
and
$$ g:\Omega G \rightarrow \bigoplus_{d \divs \order G}\Z: [G/H] \mapsto e_{\index{G}{H}}, $$
where the basis vectors \(e_d\) of the target of \(g\) are indexed by the divisors \(d\) of \(\order{G}\).
In Corollary \ref{cor:obstruction}, we show that these Brauer relations are present for any finite abelian group outside the class studied in Section \ref{sec:structure}.

\begin{lemma} \label{lem:kertrivial}
    The isomorphism class of every finite $G$-set \(X\) is uniquely determined by $\left(M_H^X\right)$ and $\left(a_d^X\right)$ if and only if $\ker(f)\cap\ker(g) = 0$.
\end{lemma}

\begin{proof}
    First, notice that for finite $G$-sets $X$ and $Y$, $f(X) = f(Y)$ if and only if $M_H^X = M_H^Y$ for every $H \sgp G$. 
    Moreover, for $X = \sum_{H \sgp G}\alpha_H^X H$, one has 
    $$ g(X) = \sum_{H\sgp G}\alpha_H^X e_{\index{G}{H}}. $$
    Recall that $\alpha_K$ counts the number of transitive $G$-sets in $X$ with stabilizer exactly $K$ for some $K\sgp G$.
    Then, an immediate consequence of the orbit-stabilizer theorem is that $a_d^X = \sum_{\index{G}{H}= d}\alpha_H^X$.
    That is, the total number of transitive $G$-sets in $X$ stabilized by a subgroup with index $d$ is exactly $a_d^X$.
    Thus $g(X) = g(Y)$ if and only if $a_d^X = a_d^Y$ for all $d$ dividing $\order{G}$. 
    Therefore, 
    $$ X-Y \in \ker(f)\cap \ker(g) $$
    if and only if
    $$ M_H^X=M_H^Y $$
    for every subgroup $H$ of $G$ and
    $$ a_d^X=a_d^Y $$
    for every divisor \(d\) of \(\order{G}\).
    Hence, every $G$-set is fully determined by $\left(M_H\right)$ and $\left(a_d\right)$ if and only if $\ker(f)\cap\ker(g) = 0$.
\end{proof}

We now provide a combinatorial formula for the rank of $\ker(f)\cap\ker(g)$. 

\begin{theorem} \label{thm:kerrank}
    Let $c_d(G)$ be the number of non-cyclic subgroups of $G$ of order $d$.
    Then the intersection of the kernels of $f$ and $g$ has rank
    $$\sum_{d| \order{G}}\max\{c_d(G) - 1,0\}. $$
\end{theorem}

\begin{proof}
    Let $D$ be the set of divisors $d$ of $\order{G}$ such that there exists a non-cyclic subgroup of $G$ with order $d$, and let $n = |D|$. 
    We first show $\rank(g(\ker(f)) = n$. Take $d \in D$, and let $H_d \sgp G$ be non-cyclic of order $d$. 
    Since $H_d$ is non-cyclic, there exist a prime $p_d$ and a subgroup $N_d \sgp H_d$ such that $\qg{H_d}{N_d} \simeq C_{p_d}\oplus C_{p_d}$.
    In particular, $\qg{H_d}{N_d}$ has the following Brauer relation (see \cite{bartel_brauer_2015}, ex. 2.4): 
    $$ \theta_{\qg{H_d}{N_d}} = 1 - \sum_{\substack{C\sgp\qg{H_d}{N_d} \\ |C| = p_d}}C + p_d[\qg{H_d}{N_d}]. $$ 
    Since $H_d \sgp G$, we can lift this to a $G$-relation (\cite{bartel_brauer_2015}, \S 2)
    $$ \theta_d = N_d - \sum_{\substack{N_d < L < H_d\\ \index{L}{N_d} = p_d}}L + p_dH_d, $$
    and compute 
    \begin{equation}
    g(\theta_d) = e_{p_d^2[G:H_d]} - (p_d + 1)e_{p_d[G:H_d]} + p_de_{[G:H_d]}. \label{eqn:gtheta}
    \end{equation}
    
    Notice that \(g(\theta_d)\) is independent of our choice of \(H_d\).
    
    We now show that 
    \begin{equation}\label{geq}
    \rank(g(\ker(f)) ) \geq n
    \end{equation} 
    by proving $\{g(\theta_d) : d \in D\}$ is a linearly independent set. 
    Assume for the sake of contradiction that the left-hand side of the equation
    \begin{equation}\label{eqn:lincombo}
            \sum_{d \in D} \lambda_d g(\theta_d) = 0
    \end{equation} 
    is a nontrivial linear combination. 
    Take the largest $d\in D$ with $\lambda_d \neq 0$ in Equation \eqref{eqn:lincombo}, 
    and observe that if \(d'<d\), then a basis vector \(e_i\) with nonzero coefficient in $g(\theta_{d'})$ must have index \(i\) greater than $\index{G}{H_d}$. 
    Hence, the coefficient of $e_{[G:H_d]}$ in the left-hand side of Equation \eqref{eqn:lincombo} is $\lambda_d p_d$, which is nonzero---a contradiction.

    Next, we construct a subgroup $H_d'$ of $G$ of order $d$, such that for any prime divisor \(p\) of \(d\), if a \(G\)-subgroup of order $d$ has a non-cyclic Sylow $p$-subgroup, then $H_d'$ also has a non-cyclic Sylow $p$-subgroup.
    We will use this subgroup to exhibit a basis of $\ker(f) \cap\ker(g)\otimes \Q$.
    
    Define $\pi(k)$ to be the set of prime divisors of $k$ and let 
    $$ G = \prod_{p\in \pi(\order{G})}G_p $$ 
    be the primary decomposition of $G$. 
    Fix $d \in D$ and let $P_{d}$ to be the set of primes dividing $d$ such that $G$ has a subgroup of order $d$ containing a non-cyclic Sylow $p$-subgroup. 
    For each $p\in\pi(d)$, choose a subgroup $A_p\sgp G_p$ of order $p^{\nu_p(d)}$, taking $A_p$ to be non-cyclic whenever $p\in P_d$. 
    Let 
    $$ H'_d = \prod_{p\in\pi(d)} A_p $$ 
    and note that $H'_d$ is non-cyclic with order $d$.
    
    Let $H \sgp G$ with $\order{H} = d$ and take some prime $p$ such that the Sylow $p$-subgroup of $H$ is non-cyclic, so that the Sylow $p$-subgroup of $H_d'$ is also non-cyclic. 
    There exist subgroups $N \sgp H$ and $N_d' \sgp H_d'$ such that 
    $$ H/N\simeq H_d'/N_d' \simeq C_{p} \oplus C_{p}. $$ 
    Let $\theta_H\tand \theta_{H'_d}$ be the corresponding $G$-relations lifted from $C_{p} \oplus C_{p}$ as before, and define $\delta_H = \theta_H - \theta_{H_d'}$.
    Since $\index{G}{H} = [G:H_d']$, we have by Equation \eqref{eqn:gtheta} that $g(\theta_H) = g(\theta_{H'_d})$, whence $\delta_H \in \ker(f) \cap \ker(g)$. 
    Then if 
    $$ S_d = \{H \sgp G  : H \textrm{ non-cyclic, } \order{H} = d, \tand  H \neq H_d'\}, $$
    one has $\order{S_d} = c_d(G) - 1,$ and $\Delta_d = \{\delta_H : H\in S_d\}$ is a linearly independent set, since $\delta_H$ is the only element in the set with a nonzero coefficient for $H$. 
    
    We now show that the union $\bigcup_{d\in D}\Delta_d$ is a linearly independent set. 
    Assume for the sake of contradiction that we have some nontrivial linear dependence 
    $$ \sum_{d\in D}\sum_{H\in S_d}\lambda_H\delta_H=0. $$
    Pick the largest $d\in D$ for which $\lambda_H \neq 0$ for some $H \in S_d$. 
    Then the only element of $\Delta_d$ for which $H$ has a nonzero coefficient is $\delta_H$. 
    For every $d_1 \in D\backslash \{d\}$, if $\lambda_{H'} \neq 0$ and $H' \in S_{d_1}$, then every term $[G/L]$ in $\delta_{H'}$ satisfies $|L| \leq d_1 < d$. Thus the coefficient of $[G/H]$ in $\delta_{H'}$ must be zero.
    But then 
    $$ \sum_{d\in D}\sum_{H\in S_d}\lambda_H\delta_H \neq 0, $$ 
    a contradiction. Hence, $\bigcup_{d\in D}\Delta_d$ is indeed linearly independent, so
    $$ \rank(\ker(f)\cap\ker(g)) \geq \sum_{d\in D}(c_d(G) - 1). $$
    
    Now, $\rank (\ker(f)) = \sum_{d\in D}c_d(G)$. 
    Applying the rank-nullity theorem to $g|_{\ker(f)}$, we find 
    $$ \rank(\ker(f)) = \rank(\ker(f)\cap \ker(g)) + \rank(g(\ker(f))), $$ 
    whence 
    \begin{equation}\label{leq}
    \rank(g(\ker(f))) \leq \sum_{d\in D}c_d(G) - \sum_{d\in D}(c_d(G) - 1) = |D| = n. 
    \end{equation}
    Together, the inequalities in \eqref{geq} and \eqref{leq} imply $\rank(g(\ker(f))) = n.$
    Hence,
    \begin{align*}
    \rank(\ker(f)\cap\ker(g)) &= \rank(\ker(f)) - \rank(g(\ker(f))) \\
    &= \sum_{d\in D}c_d(G) - n \\
    &= \sum_{d\in D}(c_d(G) -1)  \\
    &= \sum_{d| \order{G}}\max\{c_d(G) - 1,0\}
    \end{align*}
    as desired.
\end{proof}

\begin{corollary} \label{cor:obstruction}
    If there exists a divisor $d$ of $\order{G}$ such that $G$ has more than one non-cyclic subgroup of order $d$, then not every finite $G$-set is uniquely determined by \((a_d)\) and \((M_H)\).
\end{corollary}

\begin{proof}
    By Theorem \ref{thm:kerrank}, $\rank(\ker(f) \cap \ker(g)) > 0$, so by Lemma \ref{lem:kertrivial}, there are non-isomorphic $G$-sets \(X_1\) and \(X_2\) such that \(M_H^{X_1}=M_H^{X_2}\) for all \(H\sgp G\) and \(a_d^{X_1}=a_d^{X_2}\) for all \(d\) dividing \(|G|\).
\end{proof}

\section{Application to $L$-Function Factorization Rigidity}\label{sec:conclusions}

We conclude by noting that Theorem \ref{thm:main} elucidates a certain rigidity property of Artin \(L\)-function factorization.
\begin{theorem}
    Let \(E\) be a finite abelian extension of the rational numbers \(\Q\), and let \(G=\Gal(E/\Q)\) be the Galois group. 
    If \(G\) is cyclic or the direct sum of a cyclic group and a prime-order group, then for subfields \(F_1,\ldots,F_n,K_1,\ldots,K_n\) of \(E\) such that \begin{enumerate}
        \item\label{cond:equaldeg} \([F_i:\Q]=[K_i:\Q]\) for each \(i\)
        \item\label{cond:equalprod} \(\prod\limits_{i=1}^n\zeta_{F_i}(s)=\prod\limits_{i=1}^n\zeta_{K_i}(s)\),
    \end{enumerate}
    we have, after possibly reindexing, equality \(F_i=K_i\) for each \(i\).  Otherwise, there are subfields \(F_1,\ldots,F_n,K_1,\ldots,K_n\) of \(E\) satisfying Conditions \eqref{cond:equaldeg} and \eqref{cond:equalprod}, yet \(\{F_1,\ldots,F_n\}\neq\{K_1,\ldots,K_n\}\) as sets.
\end{theorem}
\begin{proof}
    Let \(X_1\) be the disjoint union of the transitive \(G\)-sets \(G/G_{F_1},\ldots,G/G_{F_n}\), and let \(\rho_1\) be the complexification of \(f(X_1)\), with \(f\) as in Equation \eqref{eqn:bside}.  
    Similarly, let \(X_2\) be the disjoint union of the \(G/G_{K_i}\), and let \(\rho_2\) be the complexification of \(f(X_2)\).  
    We then have equalities \(L(\rho_1,s)=\prod\zeta_{F_i}(s)\) and \(L(\rho_2,s)=\prod\zeta_{K_i}(s)\) of Artin \(L\)-functions. 
    Condition \eqref{cond:equalprod} amounts to an equality
    $
    L(\rho_1,s)=L(\rho_2,s),
    $
    so that \(\rho_1\simeq\rho_2\), and therefore \(|H\backslash X_1|=|H\backslash X_2|\) for every subgroup \(H\) of \(G\). Furthermore, by elementary Galois theory, Condition \eqref{cond:equaldeg} in the theorem statement amounts to \(X_1\) and \(X_2\) having the same number of \(G\)-orbits of a given size.  Hence, by Theorem \ref{thm:main}, if \(G\) is cyclic or the direct sum of a cyclic group and a prime-order group, we conclude \(X_1\simeq X_2\) as \(G\)-sets.  In particular, after relabeling if necessary, we see that \(F_i=K_i\) for each \(i\).  Conversely, if \(G\) is neither cyclic nor the direct sum of a cyclic group and a prime-order group, then as observed in the proof of Proposition \ref{prop:noncyclic}, there is a divisor \(d\) of \(|G|\) and a prime \(p\) dividing \(d\) such that \(G\) has distinct subgroups \(H_1\) and \(H_2\) of order \(d\), both of which have non-cyclic Sylow \(p\)-subgroups.  By setting \(K_1=E^{H_1}\) and \(F_1=E^{H_2}\) and lifting the respective Brauer relations (as in the proof of Theorem \ref{thm:kerrank}) on quotients of \(H_1\) and \(H_2\) isomorphic to \(C_p\oplus C_p\), we produce a family \(F_1,\ldots,F_n,K_1,\ldots,K_n\) of number fields satisfying Conditions \eqref{cond:equaldeg} and \eqref{cond:equalprod}, such that \(K_1\not\in\{F_i\}\) and \(F_1\not\in\{K_i\}\). In particular, there is no \(H_1\) term in \(\theta_{H_2}\), and there is no \(H_2\) term in \(\theta_{H_1}\), yet
    $$g(\theta_{H_1})-g(\theta_{H_2})=0$$
    and
    $$f(\theta_{H_1})-f(\theta_{H_2})=0,$$
    and terms in the last relation may be rearranged to give an equality of permutation representations.
\end{proof}
\bibliography{main_park}
\bibliographystyle{alpha}

\end{document}